\documentclass[12pt]{amsart}
\usepackage{amscd,amsmath,amssymb,enumerate,amsfonts,mathrsfs,txfonts}
\usepackage{comment}
\usepackage{hyperref}
\usepackage{xcolor}
\usepackage{tikz,pgfplots,pgf}
\usepackage[all]{xy}

\tikzset{node distance=3cm, auto}
\usetikzlibrary{calc} 
\usetikzlibrary{trees} 

\newtheorem{theorem}{Theorem}[section]
\newtheorem{proposition}[theorem]{Proposition}
\newtheorem{definition}[theorem]{Definition}
\newtheorem{corollary}[theorem]{Corollary}

\newtheorem{remark}[theorem]{Remark}

\newtheorem{examples}[theorem]{Examples}
\newcommand*\interior[1]{\mathring{#1}}

\def\A{\mathcal{A}}

\def\F{\mathcal{F}}
\def\G{\mathcal{G}}
\def\H{\mathcal{H}}
\def\I{\mathcal{I}}
\def\P{\mathcal{P}}
\def\J{\mathcal{J}}
\def\K{\mathcal{K}}
\def\L{\mathcal{L}}

\def\Nu{\mathcal{N}}
\def\Ro{\mathcal{R}}

\def\S{\mathcal{S}}

\def\U{\mathcal{U}}
\def\W{\mathcal{W}}

\def\C{\mathbb{C}}
\def\D{\mathbb{D}}

\def\N{\mathbb{N}}

\def\lin{\mathrm{lin}}

\usepackage{vmargin}
\setpapersize{A4}
\setmargins{2.5cm}      
{1.5cm}                        
{16.5cm}                      
{23.42cm}                   
{10pt}                          
{1cm}                           
{0pt}                          
{2cm}

\makeatletter
\@namedef{subjclassname@2020}{\textup{2020} Mathematics Subject Classification}
\makeatother

\begin{document}

\title[On composition ideals and dual ideals of bounded holomorphic mappings]{On composition ideals and dual ideals \\ of bounded holomorphic mappings}

\author[M. G. Cabrera-Padilla]{Mar\'ia de G\'ador Cabrera-Padilla}
\address[Mar\'ia de G\'ador Cabrera-Padilla]{Departamento de Matem\'{a}ticas, Universidad de Almer\'{i}a, 04120 Almer\'{i}a, Spain.}
\email{m\_gador@hotmail.com}

\author[A. Jim{\'e}nez-Vargas]{Antonio Jim{\'e}nez-Vargas}
\address[Antonio Jim{\'e}nez-Vargas]{Departamento de Matem{\'a}ticas, Universidad de Almer{\'i}a, 04120, Almer{\'i}a, Spain}
\email{ajimenez@ual.es}

\author[D. Ruiz-Casternado]{David Ruiz-Casternado}
\address[David Ruiz-Casternado]{Departamento de Matem{\'a}ticas, Universidad de Almer{\'i}a, 04120, Almer{\'i}a, Spain}
\email{davidrc3005@gmail.com}

\date{\today}

\subjclass[2020]{46E15, 46G20, 47B10}
\keywords{Holomorphic mapping, operator ideal, linearization, factorization theorems.}


\begin{abstract}
Applying a linearization theorem due to J. Mujica \cite{Muj-91}, we study the ideals of bounded holomorphic mappings $\H^\infty\circ \I$ generated by composition with an operator ideal $\I$. The bounded-holomorphic dual ideal of $\I$ is introduced and its elements are characterized as those that admit a factorization through $\I^\mathrm{dual}$. For complex Banach spaces $E$ and $F$, we also analyze new ideals of bounded holomorphic mappings from an open subset $U\subseteq E$ to $F$ such as $p$-integral holomorphic mappings and $p$-nuclear holomorphic mappings with $1\leq p<\infty$. We prove that every $p$-integral ($p$-nuclear) holomorphic mapping from $U$ to $F$ has relatively weakly compact (compact) range. 

\end{abstract}
\maketitle

\section*{Introduction}\label{section 0}

Let $E$ and $F$ be complex Banach spaces and let $U$ be an open subset of $E$. A mapping $f\colon U\to F$ is said to be \textit{locally compact (resp., locally weakly compact, locally Rosenthal, locally Asplund)} if every point $x\in U$ has a neighborhood $U_x\subseteq U$ such that $f(U_x)$ is relatively compact (resp., relatively weakly compact, Rosenthal, Asplund) in $F$.

R. M. Aron and M. Schottenloher \cite{AroSch-76} proved that every locally compact holomorphic mapping $f\colon E\to F$ admits a factorization of the form $f=T\circ g$, where $G$ is a Banach space, $T\colon G\to F$ belongs to the ideal of compact operators and $g\colon E\to G$ is holomorphic. Analogous results were stated by R. Ryan \cite{Rya-88}, M. Lindstr\"om \cite{Lin-89} and N. Robertson \cite{Rob-92} for locally weakly compact holomorphic mappings, locally Rosenthal holomorphic mappings and locally Asplund holomorphic mappings, respectively, with $T$ belonging to the corresponding ideal of bounded linear operators.

In \cite{GonGut-00}, M. Gonz\'alez and J. M. Guti\'errez provided a unified approach to this problem by introducing a new class of holomorphic mappings which allowed them to extend the aforementioned factorizations to every closed surjective operator ideal.  

Since many operator ideals $\I$ can be naturally associated to polynomial ideals $\P$, R. Aron, G. Botelho, D. Pellegrino and P. Rueda \cite{AroBotPelRue-10} initiated a research program whose objective was to relate those holomorphic mappings $f$ that admit a factorization $f=T\circ g$, where $T$ belongs to an operator ideal $\I$ and $g$ is holomorphic, with those $f$ whose derivatives belong to the associated composition polynomial ideal $\I\circ\P$.

Our aim is to present here some advances in this research program on the factorization of bounded holomorphic mappings $f=T\circ g$, where $T$ is in an operator ideal $\I$ and $g$ is a bounded holomorphic mapping (see Section 3 in \cite{AroBotPelRue-10}). Some recent contributions to linearization of holomorphic functions and to composition ideals are due to D. Baweja and M. Gupta \cite{BawGup-22}, and G. Botelho, V. F\'avaro and J. Mujica \cite{BotFavMuj-19}.

This paper is organized as follows. In Section \ref{section 1}, we set up the linearization theorem of bounded holomorphic mappings due to J. Mujica \cite{Muj-91} that will be a key tool to obtain our results. 

In the spirit of the definition of operator ideal, we introduce in Section \ref{section 2} the concept of ideal of (bounded) holomorphic mappings. Then, we adapt to the holomorphic setting a method to produce operator ideals from a given operator ideal $\I$ and so we obtain ideals of bounded holomorphic mappings by composition with linear operators of $\I$. This useful technique has been applied with the same purpose in different settings as, for example, polynomial and holomorphic \cite{AroBotPelRue-10}, polynomial \cite{AroRue-12}, polynomial and multilinear \cite{BotPelRue-07} and Lipschitz settings \cite{AchRueSanYah-16, Saa-17}. 

In Section \ref{section 3}, we see that some known ideals of holomorphic mappings and bounded holomorphic mappings can be produced by this procedure as, for instance, ideals of holomorphic mappings with  local range or (global) range of bounded type. 

Furthermore, we introduce and analyze new ideals of bounded holomorphic mappings such as $p$-integral and $p$-nuclear holomorphic mappings with $1\leq p<\infty$. We prove that these types of ideals are generated by the method of composition, but also give some examples of ideals of bounded holomorphic mappings that can not be produced in this way. 

The study of holomorphic mappings on Banach spaces which have relatively (weakly) compact range was initiated by J. Mujica in \cite{Muj-91} and continued by J. M. Sepulcre and the last two authors in \cite{JimRuiSep-22}. We prove here that every $p$-nuclear ($p$-integral) holomorphic mapping from $U$ to $F$ has relatively compact (weakly compact) range. 

On the other hand, the dual ideal $\I^\mathrm{dual}$ of an operator ideal $\I$ is the operator ideal formed by all bounded linear operators $T\colon E\to F$ such that its adjoint operator $T^*$ belongs to $\I(F^*,E^*)$. This procedure to produce ideals of linear operators has been applied to generate ideals of other types of mappings: polynomials and multilinear mappings between Banach spaces \cite{Pie-83, FloGar-03}, homogeneous polynomials between Banach spaces \cite{BotCalMor-14} and Lipschitz mappings from metric spaces into Banach spaces \cite{AchRueSanYah-16, Saa-17}. 

Motivated by these works and with the aid of the transpose of a bounded holomorphic mapping, we carry out in Section \ref{section 4} a similar study of the dual procedure in the setting of bounded holomorphic mappings. To be more precise, we introduce the bounded-holomorphic dual $(\I^{\H^\infty})^\mathrm{dual}$ of an operator ideal $\I$ and prove that $(\I^{\H^\infty})^\mathrm{dual}$ is not only a bounded-holomorphic ideal but also belongs to the class of composition ideals, that is, the bounded holomorphic mappings belonging to $(\I^{\H^\infty})^\mathrm{dual}$ are exactly those that admit a factorization through $\I^\mathrm{dual}$. 

We refer to the monograph by A. Pietsch \cite{Pie-80} for the theory of operator ideals, the book by J. Diestel, H. Jarchow and A. Tonge \cite{DisJarTon-95} for the theory of integral and nuclear operators, and the book by J. Mujica \cite{Muj-86} for the theory of holomorphic mappings on infinite-dimensional spaces.

\section{Preliminaries}\label{section 1}

From now on, $E$ and $F$ will denote complex Banach spaces and $U$ an open subset of $E$.

Let us recall that a mapping $f\colon U\to F$ is said to be \textit{holomorphic} if for each $a\in U$ there exist an open ball $B(a,r)$ with center at $a$ and radius $r>0$ contained in $U$ and a sequence of continuous $m$-homogeneous polynomials $(P_m)_{m\in\N_0}$ from $E$ into $F$ such that 
$$
f(x)=\sum_{m=0}^\infty P_m(x-a),
$$
where the series converges uniformly for $x\in B(a,r)$. 
A mapping $P\colon E\to F$ is said to be a \textit{continuous $m$-homogeneous polynomial} if there exists a continuous $m$-linear mapping $A\colon E\to F$ such that $P(x)=A(x,\stackrel{(m)}{\ldots},x)$ for all $x\in E$. 

If $U\subseteq E$ and $V\subseteq F$ are open sets, $\H(U,V)$ will represent the set of all holomorphic mappings from $U$ to $V$. We will denote by $\H(U,F)$ the linear space of all holomorphic mappings from $U$ into $F$, and by $\H^\infty(U,F)$ the subspace of all $f\in\H(U,F)$ such that $f(U)$ is bounded in $F$. 
In the case $F=\C$, we will write $\H(U)$ and $\H^\infty(U)$ instead of $\H(U,\C)$ and $\H^\infty(U,\C)$, respectively.

It is known that the linear space $\H^\infty(U)$, equipped with the uniform norm:
$$
\left\|f\right\|_\infty=\sup\left\{\left|f(x)\right|\colon x\in U\right\}\qquad \left(f\in\H^\infty(U)\right),
$$
is a dual Banach space. We denote by $\G^\infty(U)$ the \textit{geometric predual of $\H^\infty(U)$}. Let us recall that $\G^\infty(U)$ is the norm-closed linear hull in $\H^\infty(U)^*$ of the set $\left\{\delta(x)\colon x\in U\right\}$ of evaluation functionals defined by  
$$
\delta(x)(f)=f(x)\qquad \left(f\in\H^\infty(U)\right).
$$
The following linearization theorem by J. Mujica \cite{Muj-91} will be an essential tool to establish our results. 

Given a complex Banach space $E$, we will denote by $B_E$, $S_E$ and $E^*$ the closed unit ball, the unit sphere and the dual space of $E$, respectively. If $E$ and $F$ are Banach spaces, $\L(E,F)$ denotes the Banach space of all continuous linear operators from $E$ into $F$ with the operator canonical norm.

\begin{theorem}\cite[Theorem 2.1]{Muj-91}\label{teo0} 
Let $E$ be a complex Banach space and $U$ be an open subset of $E$. 
\begin{enumerate}
\item $\H^\infty(U)$ is isometrically isomorphic to $\G^\infty(U)^*$, under the mapping $J_U\colon\H^\infty(U)\to\G^\infty(U)^*$ given by 
$$
J_U(f)(\gamma)=\gamma(f)\qquad \left(\gamma\in\G^\infty(U),\; f\in\H^\infty(U)\right).
$$ 
\item The mapping $g_U\colon U\to\G^\infty(U)$ defined by $g_U(x)=\delta(x)$ for all $x\in U$ is holomorphic with $g_U(U)\subseteq S_{\G^\infty(U)}$. 
\item For each complex Banach space $F$ and each mapping $f\in\H^\infty(U,F)$, there exists a unique operator $T_f\in\L(\G^\infty(U),F)$ such that $T_f\circ g_U=f$. 
Furthermore, $\left\|T_f\right\|=\left\|f\right\|_\infty$. 
\item The mapping $f\mapsto T_f$ is an isometric isomorphism from $\H^\infty(U,F)$ onto $\L(\G^\infty(U),F)$.$\hfill\qed$
\end{enumerate}
\end{theorem}

We now recall some concepts and results of the theory of operator ideals which have been borrowed from \cite{DisJarTon-95, Pie-80}. 

An operator ideal $\I$ (see definition in \cite[1.1.1]{Pie-80}) is said to be:
\begin{enumerate}
\item \textit{closed} if $\I(E,F)$ is closed in $\L(E,F)$ for all Banach spaces $E$ and $F$ \cite[Section 4.2]{Pie-80}.
\item \textit{injective} if given an operator $T\in\L(E,F)$, a Banach space $G$ and an injective operator with closed range $\iota\in\L(F,G)$, we have that $T\in\I(E,F)$ whenever $\iota\circ T\in\I(E,G)$ \cite[Section 4.6]{Pie-80}.
\item \textit{surjective} if given an operator $T\in\L(E,F)$, a Banach space $G$ and a surjective operator $\pi\in\L(G,E)$, we have that $T\in\I(E,F)$ whenever $T\circ\pi\in\I(G,F)$ \cite[Section 4.7]{Pie-80}. 
\end{enumerate}

An operator $T\in\L(E,F)$ is said to be \textit{compact (resp., separable, weakly compact, Rosenthal, Asplund)} if $T(B_E)$ is relatively compact (resp., separable, relatively weakly compact, Rosenthal, Asplund) in $F$. We denote by $\F(E,F)$, $\overline{\F}(E,F)$, $\K(E,F)$, $\S(E,F)$, $\W(E,F)$, $\Ro(E,F)$ and $\A(E,F)$ the ideals of bounded finite-rank linear operators, approximable linear operators (i.e., operators which are the norm limits of bounded finite-rank operators), compact linear operators, bounded separable linear operators, weakly compact linear operators, Rosenthal linear operators and  Asplund linear operators from $E$ into $F$, respectively. The following inclusions are known:
\begin{align*}
\F(E,F)&\subseteq\overline{\F}(E,F)\subseteq\K(E,F)\subseteq\W(E,F)\subseteq\Ro(E,F)\cap\A(E,F),\\
\K(E,F)&\subseteq\S(E,F).
\end{align*}

\section{Ideal and composition ideal of bounded holomorphic mappings}\label{section 2}

Inspired by the preceding definitions, we introduce the concept of an ideal of (bounded) holomorphic mappings and its different properties.

\begin{definition}\label{def-ideal}
An ideal of holomorphic mappings (or simply, a holomorphic ideal) is a subclass $\I^{\H}$ of the class $\H$ of all holomorphic mappings such that for any complex Banach space $E$, any open subset $U$ of $E$ and any complex Banach space $F$, the components 
$$
\I^{\H}(U,F):=\I^{\H}\cap\H(U,F)
$$ 
satisfy the following three properties:
\begin{enumerate}
\item[(I1)] $\I^{\H}(U,F)$ is a linear subspace of $\H(U,F)$.
\item[(I2)] For any $g\in\H(U)$ and $y\in F$, the mapping $g\cdot y\colon x\mapsto g(x)y$ from $U$ to $F$ is in $\I^{\H}(U,F)$. 
\item[(I3)] The ideal property: If $H,G$ are complex Banach spaces, $V$ is an open subset of $H$, $h\in\H(V,U)$, $f\in\I^{\H}(U,F)$ and $S\in\L(F,G)$, then $S\circ f\circ h$ is in $\I^{\H}(V,G)$. 
\end{enumerate}

An ideal of bounded holomorphic mappings (or simply, a bounded-holomorphic ideal) is a subclass $\I^{\H^\infty}$ of $\H^\infty$ of the form $\I^{\H^\infty}=\I^{\H}\cap\H^\infty$, where $\I^{\H}$ is a holomorphic ideal. 

A bounded-holomorphic ideal $\I^{\H^\infty}$ is said to be:
\begin{enumerate}
	\item closed if each component $\I^{\H^\infty}(U,F)$ is a closed subspace of $\H^\infty(U,F)$ with the topology of supremum norm.
	\item injective if for any mapping $f\in\H^\infty(U,F)$, any complex Banach space $G$ and any isometric linear embedding $\iota\colon F\to G$, we have that $f\in\I^{\H^\infty}(U,F)$ whenever $\iota\circ f\in\I^{\H^\infty}(U,G)$. 
	\item surjective if for any mapping $f\in\H^\infty(U,F)$, any open subset $V$ of a complex Banach space $G$ and any surjective mapping $\pi\in\H(V,U)$, we have that $f\in\I^{\H^\infty}(U,F)$ whenever $f\circ \pi\in\I^{\H^\infty}(V,F)$.
\end{enumerate}
 
A bounded-holomorphic ideal $\I^{\H^\infty}$ is said to be normed (Banach) if there exists a function $\left\|\cdot\right\|_{\I^{\H^\infty}}\colon\I^{\H^\infty}\to\mathbb{R}_0^+$ such that for every complex Banach space $E$, every open subset $U$ of $E$ and every complex Banach space $F$, the following three conditions are satisfied:
\begin{enumerate}
	\item[(N1)] $(\I^{\H^\infty}(U,F),\left\|\cdot\right\|_{\I^{\H^\infty}})$ is a normed (Banach) space with $	\left\|f\right\|_\infty\leq\left\|f\right\|_{\I^{\H^\infty}}$ for all $f\in\I^{\H^\infty}(U,F)$.
	\item[(N2)] $\left\|g\cdot y\right\|_{\I^{\H^\infty}}=\left\|g\right\|_\infty\left\|y\right\|$ for all $g\in\H^\infty(U)$ and $y\in F$. 
	\item[(N3)] If $H,G$ are complex Banach spaces, $V$ is an open subset of $H$, $h\in\H(V,U)$, $f\in\I^{\H^\infty}(U,F)$ and $S\in\L(F,G)$, then $\left\|S\circ f\circ h\right\|_{\I^{\H^\infty}}\leq \left\|S\right\|\left\|f\right\|_{\I^{\H^\infty}}$.
\end{enumerate}

\end{definition}

\begin{remark}
According to Definition \ref{def-ideal}, ideals of holomorphic mappings always contain the finite rank maps. So, when restricted to the case of homogeneous polynomials, our definition does not recover the classical notion of polynomial ideals (a polynomial ideal contains the finite type maps, and not necessarily the finite rank maps). We must point out that our notion of ideal of holomorphic mappings is more related to the notion of hyper-ideals of polynomials (which always contain the finite rank maps) rather than to the notion of polynomial ideals. For the theory of hyper-ideals of polynomials, see, e.g., \cite{BotWoo-23}. 
\end{remark}

We now recall the composition method to produce ideals of holomorphic mappings. This linear method has been applied in different contexts: multilinear, polynomial, Lipschitz and holomorphic (see, for example, \cite{AchRueSanYah-16, AroBotPelRue-10, BotPelRue-07,GonGut-00, Saa-17}). 

\begin{definition}\label{comp-ideal}
Let $E,F$ be complex Banach spaces and $U$ be an open set in $E$. Given an operator ideal $\I$, a mapping $f\in\H(U,F)$ (resp. $f\in\H^\infty(U,F)$) belongs to the composition ideal $\I\circ\H$ (resp. $\I\circ\H^\infty$), and we write $f\in\I\circ\H(U,F)$ (resp. $f\in\I\circ\H^\infty(U,F)$), if there are a complex Banach space $G$, an operator $T\in\I(G,F)$ and a mapping $g\in\H(U,G)$ (resp. $g\in\H^\infty(U,G)$) such that $f=T\circ g$. 

If $(\I,\left\|\cdot\right\|_\I)$ is a normed operator ideal and $f\in\I\circ\H^\infty$, we denote 
$$
\left\|f\right\|_{\I\circ\H^\infty}=\inf\left\{\left\|T\right\|_\I\left\|g\right\|_\infty\right\},
$$
where the infimum is extended over all representations of $f$ as above.
\end{definition}

The following result states the linearization of the members of the composition ideal $\I\circ\H^\infty$. The first part was proved in \cite{AroBotPelRue-10}, but we have included it here for the convenience of the reader.

\begin{theorem}\cite[Theorem 3.2]{AroBotPelRue-10}\label{ideal}
Let $\I$ be an operator ideal and $f\in\H^\infty(U,F)$. The following conditions are equivalent:
\begin{enumerate}
	\item $f$ belongs to $\I\circ\H^\infty(U,F)$.
	\item Its linearization $T_f$ is in $\I(\G^\infty(U),F)$.
\end{enumerate}
If $(\I,\left\|\cdot\right\|_\I)$ is a normed operator ideal, we have that $\left\|f\right\|_{\I\circ\H^\infty}=||T_f||_\I$ and the infimum $\left\|f\right\|_{\I\circ\H^\infty}$ is attained at $T_f\circ g_U$ (\textit{Mujica's factorization of $f$}). Furthermore, the mapping $f\mapsto T_f$ is an isometric isomorphism from $(\I\circ\H^\infty(U,F),\left\|\cdot\right\|_{\I\circ\H^\infty})$ onto $(\I(\G^\infty(U),F),\left\|\cdot\right\|_\I)$. 
\end{theorem}

\begin{proof}
$(1)\Rightarrow(2)$: If $f\in\I\circ\H^\infty(U,F)$, then there are a complex Banach space $G$, a mapping $g\in\H^\infty(U,G)$ and an operator $T\in\I(G,F)$ such that $f=T\circ g$. Since $f=T_f\circ g_U$ and $g=T_g\circ g_U$ by Theorem \ref{teo0}, it follows that $T_f\circ g_U=T\circ T_g\circ g_U$ which implies that $T_f=T\circ T_g$ by the linear denseness of $ g_U(U)$ in $\G^\infty(U)$, and thus $T_f\in\I(\G^\infty(U),F)$ by the ideal property of $\I$. Further, if $(\I,\left\|\cdot\right\|_\I)$ is normed, we have
$$
\left\|T_f\right\|_\I=\left\|T\circ T_g\right\|_\I\leq\left\|T\right\|_\I\left\|T_g\right\|=\left\|T\right\|_\I\left\|g\right\|_\infty, 
$$
and taking the infimum over all representations of $f$, we deduce that $\left\|T_f\right\|_\I\leq\left\|f\right\|_{\I\circ\H^\infty}$. 

$(2)\Rightarrow(1)$: If $T_f\in\I(\G^\infty(U),F)$, then $f=T_f\circ g_U\in\I\circ\H^\infty(U,F)$ since $\G^\infty(U)$ is a complex Banach space and $ g_U\in\H^\infty(U,\G^\infty(U))$ by Theorem \ref{teo0}. Moreover, if $(\I,\left\|\cdot\right\|_\I)$ is normed, we have 
$$
\left\|f\right\|_{\I\circ\H^\infty}=\left\|T_f\circ g_U\right\|_{\I\circ\H^\infty}\leq\left\|T_f\right\|_\I\left\| g_U\right\|_\infty=\left\|T_f\right\|_\I. 
$$

Finally, the last assertion of the statement easily follows by applying the above proof and Theorem \ref{teo0}.
\end{proof}

We now see that some properties of the operator ideal $\I$ are transferred to the composition ideal $\I\circ\H^\infty$.

\begin{corollary}\label{new2}
If $\I$ is a closed (resp., normed, Banach) operator ideal, then $\I\circ\H^\infty$ is a closed (resp., normed, Banach) bounded-holomorphic ideal.
\end{corollary}

\begin{proof}
Let $\I$ be an operator ideal. We have:
\begin{enumerate}
\item[(I1)] If $\alpha_1,\alpha_2\in\mathbb{C}$ and $f_1,f_2\in\I\circ\H^\infty(U,F)$, then $T_{\alpha_1f_1+\alpha_2f_2}=\alpha_1T_{f_1}+\alpha_2T_{f_1}\in\I(\G^\infty(U),F)$ by Theorems \ref{teo0} and \ref{ideal}. Hence $\alpha_1f_1+\alpha_2f_2\in\I\circ\H^\infty(U,F)$ by Theorem \ref{ideal}. Therefore $\I\circ\H^\infty(U,F)$ is a linear subspace of $\H^\infty(U,F)$. 
\item[(I2)] Given $g\in\H^\infty(U)$ and $y\in F$, we can write $g\cdot y=T_{g\cdot y}\circ g_U$, where $ g_U\in\H^\infty(U,G^\infty(U))$ and $T_{g\cdot y}\in\F(G^\infty(U),F)$ by Theorem \ref{teo0} and \cite[p. 872]{Muj-91}, and this tells us that $g\cdot y\in \F\circ\H^\infty(U,F)$ and since always $\F\subseteq\I$, we conclude that $g\cdot y\in\I\circ\H^\infty(U,F)$. 
\item[(I3)] Let $H,G$ be complex Banach spaces, $V$ be an open subset of $H$, $h\in\H(V,U)$, $f\in\I\circ\H^\infty(U,F)$ and $S\in\L(F,G)$. By \cite[Corollary 1.4]{JimRuiSep-22}, there exists a unique operator $\widehat{h}\in\L(\G^\infty(V),\G^\infty(U))$ such that $\widehat{h}\circ g_V=g_U\circ h$. Furthermore, $||\widehat{h}||=1$. Since  
$$
S\circ f\circ h=S\circ(T_f\circ g_U)\circ h=(S\circ T_f\circ\widehat{h})\circ g_V,
$$
with $S\circ T_f\circ\widehat{h}\in\I(\G^\infty(V),G)$ and $g_V\in\H^\infty(V,\G^\infty(V))$, we have $S\circ f\circ h\in\I\circ\H^\infty(V,G)$.
\end{enumerate}
This proves that $\I\circ\H^\infty$ is a bounded-holomorphic ideal.

We now show that $\I\circ\H^\infty(U,F)$ is closed whenever $\I$ is so. Let $f\in\I\circ\H^\infty(U,F)$ and let $(f_n)_{n\in\N}$ be a sequence in $\I\circ\H^\infty(U,F)$ such that $\left\|f_n-f\right\|_\infty\to 0$ as $n\to\infty$. Since $T_{f_n}\in\I(G^\infty(U),F)$ by Theorem \ref{ideal} and $\left\|T_{f_n}-T_f\right\|=\left\|T_{f_n-f}\right\|=\left\|f_n-f\right\|_\infty$ for all $n\in\N$, we have that $T_f\in\I(\G^\infty(U),F)$, and thus $f\in\I\circ\H^\infty(U,F)$ by Theorem \ref{ideal}.

Assume now that the operator ideal $(\I,\left\|\cdot\right\|_\I)$ is normed. We have:
\begin{enumerate}
\item[(N1)] Since $\left\|f\right\|_{\I\circ\H^\infty}=\left\|T_f\right\|_\I$ for all $f\in\I\circ\H^\infty(U,F)$ by Theorem \ref{ideal}, it easily follows that $\left\|\cdot\right\|_{\I\circ\H^\infty}$ is a norm on $\I\circ\H^\infty(U,F)$ and 
$$
\left\|f\right\|_\infty=\left\|T_f\right\|\leq \left\|T_f\right\|_\I=\left\|f\right\|_{\I\circ\H^\infty}
$$ 
for all $f\in\I\circ\H^\infty(U,F)$. 
\item[(N2)] Given $g\in\H^\infty(U)$ and $y\in F$, we have 
$$
\left\|g\right\|_\infty\left\|y\right\|=\left\|g\cdot y\right\|_\infty\leq \left\|g\cdot y\right\|_{\I\circ\H^\infty},
$$
and conversely, since $g\cdot y=M_y\circ g$ where $M_y\in\F(\mathbb{C},F)\subseteq\I(\mathbb{C},F)$ is the operator defined $M_y(\lambda)=\lambda y$ for all $\lambda\in\mathbb{C}$, we have
$$
\left\|g\cdot y\right\|_{\I\circ\H^\infty}\leq\left\|g\right\|_\infty\left\|M_y\right\|=\left\|g\right\|_\infty\left\|y\right\|.
$$
\item[(N3)] Following the above proof of (I3), we have 
\begin{align*}
\left\|S\circ f\circ h\right\|_{\I\circ\H^\infty}&=\left\|(S\circ T_f\circ\widehat{h})\circ g_V\right\|_{\I\circ\H^\infty}\leq \left\|S\circ T_f\circ\widehat{h}\right\|_\I \left\|g_V\right\|_\infty\\
																								 &\leq \left\|S\right\|\left\|T_f\right\|_\I\left\|\widehat{h}\right\|\left\|g_V\right\|_\infty=\left\|S\right\|\left\|f\right\|_{\I\circ\H^\infty}.
\end{align*}
\end{enumerate}
So, we have proved that the ideal $(\I\circ\H^\infty,\left\|\cdot\right\|_{\I\circ\H^\infty})$ is normed.  

Finally, since $(\I\circ\H^\infty(U,F),\left\|\cdot\right\|_{\I\circ\H^\infty})$ is isometrically isomorphic to $(\I(\G^\infty(U),F),\left\|\cdot\right\|_\I)$ by Theorem \ref{ideal}, then $(\I\circ\H^\infty,\left\|\cdot\right\|_{\I\circ\H^\infty})$ is a Banach ideal whenever $(\I,\left\|\cdot\right\|_\I)$ is so. 
\end{proof}

We finish this section with some properties of bounded-holomorphic ideals which can be easily deduced from Theorems \ref{teo0} and \ref{ideal}.

\begin{proposition}
Let $\I$ and $\J$ be two operator ideals. We have:
\begin{enumerate}
	\item If $\I\circ\H^\infty(U,F)\subseteq\J\circ\H^\infty(U,F)$, then $\I(G^\infty(U),F)\subseteq\J(G^\infty(U),F)$.
	\item If $\I\circ\H^\infty(U,F)=\H^\infty(U,F)$, then $\I(G^\infty(U),F)=\L(G^\infty(U),F)$.
	\item If the identity operator $\mathrm{id}_F\in\I(F,F)$, then $\I\circ\H^\infty(U,F)=\H^\infty(U,F)$. $\hfill\qed$
\end{enumerate}
\end{proposition}


\section{Examples of composition ideals of bounded holomorphic mappings}\label{section 3}

We have divided this section into four parts. In Subsection \ref{subsection 1}, we will recall that some known ideals of holomorphic mappings $f$ with local range of bounded type (for example, with  compact, weakly compact, Rosenthal or Asplund range) are generated by composition of a holomorphic mapping $g$ with an operator $T$ in the corresponding operator ideal $\I$. 

If $f$ is in addition bounded, one cannot assure in general that the function $g$ is also bounded. However, this is possible if we consider some smaller classes of such ideals. To be more precise, in Subsection \ref{subsection 2} we will show that the ideals of bounded holomorphic mappings with (global) range of bounded type are generated by composition of a bounded holomorphic mapping $g$ with an operator $T$ in $\I$. 

Finally, motivated by the ideals of $p$-integral operators and $p$-nuclear operators between Banach spaces for $1\leq p<\infty$ (see, for example, \cite{DisJarTon-95}), we will introduce in Subsections \ref{subsection 3} and \ref{subsection 4} the analogs in the holomorphic setting and state that such ideals of bounded holomorphic mappings are generated by the method of composition.

\subsection{Holomorphic mappings with local range of bounded type}\label{subsection 1}


Let $\H_{k}(U,F)$ (resp., $\H_{w}(U,F)$, $\H_{r}(U,F)$, $\H_{a}(U,F)$) denote the linear subspace of all locally compact (resp., locally weakly compact, locally Rosenthal, locally Asplund) mappings of $\H(U,F)$. In the bounded case, we write
$$
\H^\infty_i(U,F)=\H_i(U,F)\cap\H^\infty(U,F)\qquad (i=k,w,r,a).
$$ 
It is clear that 
$$
\H_k(U,F)\subseteq\H_w(U,F)\subseteq\H_r(U,F)\cap\H_a(U,F).
$$

\begin{proposition}\label{Hi}
For $i=k,w,r,a$, $\H_i$ is a holomorphic ideal and $\H^\infty_i$ is a bounded-holomorphic ideal.
\end{proposition}

\begin{proof}
Let $i=k,w,r,a$. Clearly, $\H_i(U,F)$ is a linear subspace of $\H(U,F)$. Given $g\in\H(U)$ and $y\in F$, it is clear that $g\cdot y\in\H(U,F)$ with $(g\cdot y)(U)=g(U)y$. Since $g$ is locally bounded by \cite[Lemma 5.6]{Muj-86}, every point $x\in U$ has a neighborhood $U_x\subseteq U$ such that $g(U_x)$ is bounded in $\C$, that is, relatively compact in $\C$. Hence $(g\cdot y)(U_x)=g(U_x)y$ is relatively compact in $F$ and thus $g\cdot y\in\H_k(U,F)\subseteq\H_i(U,F)$. 

To prove the ideal property of $\H_i(U,F)$, let $H,G$ be complex Banach spaces, $V$ be an open subset of $H$, $h\in\H(V,U)$, $f\in\H_i(U,F)$ and $S\in\L(F,G)$. Let $x\in V$. Then there exists a neighborhood of $h(x)$, $U_{h(x)}\subseteq U$, such that $f(U_{h(x)})$ is relatively compact (resp., relatively weakly compact, Rosenthal, Asplund) in $F$. Denote $V_x=h^{-1}(U_{h(x)})$. Hence $S(f(h(V_x))$ is relatively compact (resp., relatively weakly compact, Rosenthal, Asplund) in $G$, and thus $S\circ f\circ h\in\H_i(V,G)$. This proves that $\H_i$ is a holomorphic ideal. Hence $\H^\infty_i$ is a bounded-holomorphic ideal. 

\end{proof}

Some known results show that the holomorphic ideals $\H_i$ for $i=k,w,r,a$ are generated by the method of composition with an operator ideal. Namely, we have
\begin{align*}
\K\circ\H(E,F)=H_k(E,F)\qquad&\text{\cite{AroSch-76}},\\
\W\circ\H(E,F)=H_w(E,F)\qquad&\text{\cite{Rya-88}},\\
\Ro\circ\H(E,F)=H_r(E,F)\qquad&\text{\cite{Lin-89}},\\
\A\circ\H(E,F)=H_a(E,F)\qquad&\text{\cite{Rob-92}}.
\end{align*}
More generally, let $\U$ be a closed surjective operator ideal and let $\mathcal{C}_{\U}(F)$ be the collection of all $A\subseteq F$ so that $A\subseteq T(B_G)$ for some complex Banach space $G$ and some operator $T\in\U(G,F)$. Let $H^\U(E,F)$ denote the space of all $f\in\H(E,F)$ such that each $x\in E$ has a neighborhood $V_x\subseteq E$ with $f(V_x)\in\mathcal{C}_\U(F)$. In \cite[Theorem 6]{GonGut-00}, it is proved that
$$
\U\circ\H(E,F)=H^\U(E,F).
$$
See also the paper \cite{AroBotPelRue-10} for a study of such spaces as associated to composition ideals of polynomials.

\subsection{Holomorphic mappings with range of bounded type}\label{subsection 2}


Let $\H^\infty_{\K}(U,F)$ (resp., $\H^\infty_{\W}(U,F)$, $\H^\infty_{\Ro}(U,F)$, $\H^\infty_{\A}(U,F)$) denote the linear space of all holomorphic mappings $f\colon U\to F$ such that $f(U)$ is relatively compact (resp., relatively weakly compact, Rosenthal, Asplund) in $F$. Note that $f(U)$ is a bounded subset of $F$ in all the cases. 

We will also consider the spaces: 
\begin{align*}
	\H^\infty_{\F}(U,F)&=\left\{f\in\H^\infty(U,F)\colon \lin(f(U))\text{ is finite-dimensional in }F\right\},\\
	\H^\infty_{\overline{\F}}(U,F)&=\left\{f\in\H^\infty(U,F)\colon \exists (f_n)_n^\infty\subseteq\H_\F^\infty(U,F)\;|\; \left\|f_n-f\right\|_\infty\to 0\right\},\\
	\H^\infty_{\S}(U,F)&=\left\{f\in\H^\infty(U,F)\colon f(U)\text{ is separable in }F\right\}.
\end{align*}
Clearly, we have the inclusions:
\begin{align*}
\H^\infty_{\F}(U,F)&\subseteq\H^\infty_{\overline{\F}}(U,F)\subseteq\H^\infty_\K(U,F)\subseteq\H^\infty_\W(U,F)\subseteq\H^\infty_\Ro(U,F)\cap\H^\infty_\A(U,F),\\
\H^\infty_{\K}(U,F)&\subseteq\H^\infty_{\S}(U,F),
\end{align*}
and
$$
\H^\infty_{\I}(U,F)\subseteq\H^\infty_{i}(U,F)\qquad ((\I,i)=(\K,k),(\W,w),(\Ro,r),(\A,a)).
$$

\begin{proposition}\label{HI}
For $\I=\F,\K,\S,\overline{\F},\W,\Ro,\A$, the set $\H^\infty_\I$ is a bounded-holomorphic ideal. Furthermore, we have:
\begin{enumerate}
	\item $\H^\infty_{\I}$ is not closed if and only if $\I=\F$. 
	\item $\H^\infty_{\I}$ is injective and surjective whenever $\I=\K,\S,\W,\Ro,\A$.
\end{enumerate}
\end{proposition}

\begin{proof}
The first assertion follows with a proof similar to that of Proposition \ref{Hi}. Applying that the mapping $f\mapsto T_f$ is an isometric isomorphism from $\H^\infty_{\I}(U,F)$ onto $\I(\G^\infty(U),F)$ (see \cite[Propositions 3.1 and 3.4]{Muj-91} for $\I=\F,\K,\W$ and \cite[Corollary 2.11]{JimRuiSep-22} for $\I=\S,\Ro,\A$), and that $\I$ is not closed for the operator norm if and only if $\I=\F$, we deduce the equivalence in $(1)$.  

Let $\I=\K,\S,\W,\Ro,\A$ and $f\in\H^\infty(U,F)$. On the one hand, assume that $\iota\circ f\in\H^\infty_{\I}(U,G)$ for any complex Banach space $G$ and any isometric linear embedding $\iota\colon F\to G$. Since $\iota\circ T_f=T_{\iota\circ f}\in\I(\G^\infty(U),G)$ by \cite[Propositions 3.1 and 3.4]{Muj-91} and \cite[Corollary 2.11]{JimRuiSep-22}, and the operator ideal $\I$ is injective (see, for example, \cite[p. 471]{GonGut-00}), it follows that $T_f\in\I(\G^\infty(U),F)$, thus $f\in\H_{\I}^\infty(U,F)$ by \cite[Propositions 3.1 and 3.4]{Muj-91} and \cite[Corollary 2.11]{JimRuiSep-22}, and this proves that $\H^\infty_{\I}$ is injective. 

On the other hand, suppose that $f\circ \pi\in\I^{\H^\infty}(V,F)$, where $V$ is an open subset of a complex Banach space $G$ and $\pi\in\H(V,U)$ is surjective. By \cite[Corollary 1.4]{JimRuiSep-22}, there exists a unique operator $\widehat{\pi}\in\L(\G^\infty(V),\G^\infty(U))$ such that $g_U\circ\pi=\widehat{\pi}\circ g_V$. Since $T_f\circ \widehat{\pi}\in\L(\G^\infty(V),F)$ and $T_f\circ\widehat{\pi}\circ g_V=T_f\circ g_U\circ\pi=f\circ\pi$, we deduce that $T_{f\circ\pi}=T_f\circ\widehat{\pi}$ by Theorem \ref{teo0}. Since $T_f\circ\widehat{\pi}=T_{f\circ \pi}\in\I(\G^\infty(V),F)$ by \cite[Propositions 3.1 and 3.4]{Muj-91} and \cite[Corollary 2.11]{JimRuiSep-22}, and the operator ideal $\I$ is surjective \cite[p. 471]{GonGut-00}, we have that $T_f\in\I(\G^\infty(U),F)$, hence $f\in\H_{\I}^\infty(U,F)$ by \cite[Propositions 3.1 and 3.4]{Muj-91} and \cite[Corollary 2.11]{JimRuiSep-22}, and thus $\H^\infty_{\I}$ is surjective.
\end{proof}

We next see that the preceding bounded-holomorphic ideals are generated by composition with the corresponding operator ideal. 

\begin{proposition}\label{new}
For $\I=\F,\overline{\F},\K,\S,\W,\Ro,\A$, we have $\H^\infty_{\I}=\I\circ\H^\infty$ and $\left\|f\right\|_\infty=\left\|f\right\|_{\I\circ\H^\infty}$ for all $f\in\H^\infty_{\I}$. 
\end{proposition}

\begin{proof}
Let $\I=\F,\overline{\F},\K,\S,\W,\Ro,\A$. Note first that $\H^\infty_{\I}(U,F)=\I\circ\H^\infty(U,F)$. Indeed, if $f\in\H^\infty_{\I}(U,F)$, then $f=T_f\circ g_U$ where $T_f\in\I(\G^\infty(U),F)$ by \cite[Propositions 3.1 and 3.4]{Muj-91} and \cite[Corollary 2.11]{JimRuiSep-22}, and thus $f\in\I\circ\H^\infty(U,F)$. Conversely, if $f\in\I\circ\H^\infty(U,F)$, then $f=T\circ g$ for some complex Banach space $G$, $g\in\H^\infty(U,G)$ and $T\in\I(G,F)$. If $\I=\F$ or $\I=\overline{\F}$, then $f$ has a finite rank or $f$ can be approximated by bounded finite-rank holomorphic mappings, respectively. If $\I=\K,\S,\W,\Ro,\A$, since $g(U)$ is bounded in $G$, it follows that $f(U)=T(g(U))$ is relatively compact (resp., separable, relatively weakly compact, Rosenthal, Asplund) in $F$. Hence $f\in\H^\infty_{\I}(U,F)$, as required.

Now, if $\left\|\cdot\right\|_\I$ denotes the operator norm, recall that the mappings 
\begin{align*}
f\in(\H^\infty_{\I}(U,F),\left\|\cdot\right\|_\infty)&\mapsto T_f\in(\I(\G^\infty(U),F),\left\|\cdot\right\|_\I),\\
f\in(\I\circ\H^\infty(U,F),\left\|\cdot\right\|_{\I\circ\H^\infty})&\mapsto T_f\in(\I(\G^\infty(U),F),\left\|\cdot\right\|_\I),
\end{align*}
are isometric isomorphisms (see \cite[Propositions 3.1 and 3.4]{Muj-91} and \cite[Corollary 2.11]{JimRuiSep-22} for the first, and Theorem \ref{ideal} for the second). Hence we have
$$
\left\|f\right\|_\infty=\left\|T_f\right\|_\I=\left\|f\right\|_{\I\circ\H^\infty}
$$
for all $f\in\H^\infty_{\I}(U,F)$. 
\end{proof}

\subsection{$p$-integral holomorphic mappings}\label{subsection 3}

Following \cite[p. 93]{DisJarTon-95}, given two Banach spaces $E$, $F$ and $1\leq p\leq\infty$, we denote by $\I_p(E,F)$ the Banach space of all $p$-integral linear operators $T\colon E\to F$ with the norm 
$$
\iota_p(T)=\inf\left\{\left\|A\right\|\left\|B\right\|\right\},
$$
where the infimum is taken over all $p$-integral factorizations $(A,I^\mu_{\infty,p},B)$ of $T$ in the form
$$
\kappa_F\circ T=A\circ I^\mu_{\infty,p}\circ B\colon E\stackrel{B}{\rightarrow}L_\infty(\mu)\stackrel{I^\mu_{\infty,p}}{\rightarrow}L_p(\mu)\stackrel{A}{\rightarrow}F^{**},
$$
where $(\Omega,\Sigma,\mu)$ is a probability measure space, $A\in\L(L_p(\mu),F^{**})$ and $B\in\L(E,L_\infty(\mu))$. As usual, $I^\mu_{\infty,p}\colon L_\infty(\mu)\to L_p(\mu)$ is the formal identity, and $\kappa_F\colon F\to F^{**}$ is the canonical isometric embedding.  

Let $p^*$ denote the conjugate index of $p\in [1,\infty]$ defined by $p^*=p/(p-1)$ if $p\neq 1$, $p^*=\infty$ if $p=1$, and $p^*=1$ if $p=\infty$. 

The concept of $p$-integral linear operator motivates us to introduce the holomorphic analog as follows.

\begin{definition}\label{def-bsLio}
Let $E,F$ be complex Banach spaces, $U$ be an open subset of $E$ and $1\leq p\leq\infty$. A mapping $f\colon U\to F$ is said to be $p$-integral holomorphic if there exist a probability measure space $(\Omega,\Sigma,\mu)$, an operator $T\in\L(L_p(\mu),F^{**})$ and a mapping $g\in\H^\infty(U,L_\infty(\mu))$ giving rise to the commutative diagram:
$$
\begin{tikzpicture}
  \node (U) {$U$};
  \node (F) [right of=U] {$F$};
  \node (F**) [right of=F] {$F^{**}$};
  \node (L) [below of=U] {$L_\infty(\mu)$};
  \node (Lp) [below of=F**] {$L_p(\mu)$};
  \draw[->] (U) to node {$f$} (F);
  \draw[->] (F) to node {$\kappa_F$} (F**);
  \draw[->] (U) to node [swap] {$g$} (L);
  \draw[->] (L) to node {$I^\mu_{\infty,p}$} (Lp);
  \draw[->] (Lp) to node [swap] {$T$} (F**);
\end{tikzpicture}
$$
The triple $(T,I^\mu_{\infty,p},g)$ is called a $p$-integral holomorphic factorization of $f$. We denote 
$$
\iota^{\H^\infty}_p(f)=\inf\left\{\left\|T\right\|\left\|g\right\|_\infty\right\},
$$
where the infimum is extended over all such factorizations of $f$. Let $\I^{\H^\infty}_p(U,F)$ denote the set of all $p$-integral holomorphic mappings from $U$ into $F$.
\end{definition}

We now study the linearization of $p$-integral holomorphic mappings.

\begin{proposition}\label{integral}
Let $1\leq p\leq\infty$ and $f\in\H^\infty(U,F)$. Then $f\colon U\to F$ is $p$-integral holomorphic if and only if its linearization $T_f\colon\G^\infty(U)\to F$ is $p$-integral. In this case,
$$
\iota_p(T_f)=\iota_p^{\H^\infty}(f).
$$
Furthermore, the mapping $f\mapsto T_f$ is an isometric isomorphism from $(\I_p^{\H^\infty}(U,F),\iota_p^{\H^\infty})$ onto $(\I_p(\G^\infty(U),F),\iota_p)$. 
\end{proposition}

\begin{proof}
If $f\colon U\to F$ is $p$-integral holomorphic, then we have 
$$
\kappa_F\circ f=T\circ I^\mu_{\infty,p}\circ g\colon U\stackrel{g}{\rightarrow}L_\infty(\mu)\stackrel{I^\mu_{\infty,p}}{\rightarrow}L_p(\mu)\stackrel{T}{\rightarrow}F^{**},
$$
where $(T,I^\mu_{\infty,p},g)$ is a $p$-integral holomorphic factorization of $f$. Applying Theorem \ref{teo0}, we obtain 
$$
\kappa_F\circ T_f\circ g_U=T\circ I^\mu_{\infty,p}\circ T_g\circ g_U\colon U\stackrel{ g_U}{\rightarrow}\G^\infty(U)\stackrel{T_g}
{\rightarrow}L_\infty(\mu)\stackrel{I^\mu_{\infty,p}}{\rightarrow}L_p(\mu)\stackrel{T}{\rightarrow}F^{**},
$$
By the denseness of $\lin(g_U(U))$ in $\G^\infty(U)$, we deduce that
$$
\kappa_F\circ T_f=T\circ I^\mu_{\infty,p}\circ T_g\colon\G^\infty(U)\stackrel{T_g}{\rightarrow}L_\infty(\mu)\stackrel{I^\mu_{\infty,p}}{\rightarrow}L_p(\mu)\stackrel{T}{\rightarrow}F^{**}
$$
and therefore $T_f\colon\G^\infty(U)\to F$ is $p$-integral. Furthermore, we have
$$
\iota_p(T_f)\leq\left\|T\right\|\left\|T_g\right\|=\left\|T\right\|\left\|g\right\|_\infty
$$
and taking infimum over all the $p$-integral holomorphic factorization of $f$, we deduce
$$
\iota_p(T_f)\leq\iota_p^{\H^\infty}(f).
$$
Conversely, if $T_f\colon\G^\infty(U)\to F$ is $p$-integral, we have 
$$
\kappa_F\circ T_f=T\circ I^\mu_{\infty,p}\circ S\colon\G^\infty(U)\stackrel{S}{\rightarrow}L_\infty(\mu)\stackrel{I^\mu_{\infty,p}}{\rightarrow}L_p(\mu)\stackrel{T}{\rightarrow}F^{**}
$$
where $(T,I^\mu_{\infty,p},S)$ is a $p$-integral factorization of $T_f$. Note that $g:=S\circ g_U\in\H^\infty(U,L_\infty(\mu))$, and since  
$$
\kappa_F\circ f=T\circ I^\mu_{\infty,p}\circ g\colon U\stackrel{g}{\rightarrow}L_\infty(\mu)\stackrel{I^\mu_{\infty,p}}{\rightarrow}L_p(\mu)\stackrel{T}{\rightarrow}F^{**}
$$ 
we conclude that $f$ is $p$-integral holomorphic. Furthermore, we have  
$$
\iota_p^{\H^\infty}(f)\leq\left\|T\right\|\left\|g\right\|_\infty=\left\|T\right\|\left\|S\right\|,
$$
and taking infimum over all the $p$-integral factorization of $T_f$, we deduce
$$
\iota_p^{\H^\infty}(f)\leq\iota_p(T_f).
$$
To prove the last assertion of the statement, it suffices to show that the mapping $f\mapsto T_f$ from $\I_p^{\H^\infty}(U,F)$ to $\I_p(\G^\infty(U),F)$ is surjective. Take $T\in\I_p(\G^\infty(U),F)$ and then $T=T_f$ for some $f\in\H^\infty(U,F)$ by Theorem \ref{teo0}. Hence $T_f\in\I_p(\G^\infty(U),F)$ and this implies that $f\in\I_p^{\H^\infty}(U,F)$ by the above proof. 
\end{proof}

Combining Theorem \ref{ideal} and Proposition \ref{integral}, we deduce that the bounded-holomorphic ideal $\I_p^{\H^\infty}$ is generated by composition with the operator ideal $\I_p$.

\begin{corollary}\label{now}
Let $1\leq p\leq\infty$. Then $\I_p^{\H^\infty}=\I_p\circ\H^\infty$ and $\iota_p^{\H^\infty}(f)=\left\|f\right\|_{\I_p\circ\H^\infty}$ for all $f\in\I_p^{\H^\infty}$. In particular, $\left(\I^{\H^\infty}_p,\iota^{\H^\infty}_p\right)$ is a Banach ideal of bounded holomorphic mappings. $\hfill\qed$
\end{corollary}

We next see that if a mapping is 1-integral holomorphic, then it is $p$-integral holomorphic for any $1\leq p\leq\infty$. 

\begin{corollary}
Let $1\leq p<q\leq\infty$. Then $\I^{\H^\infty}_p(U,F)\subseteq\I^{\H^\infty}_q(U,F)$ and $\iota^{\H^\infty}_q(f)\leq\iota^{\H^\infty}_p(f)$ for each $f\in\I^{\H^\infty}_p(U,F)$.
\end{corollary}

\begin{proof}
It follows immediately from Proposition \ref{integral} and \cite[Proposition 5.1]{DisJarTon-95}.
\end{proof}

We finish our study of $p$-integral holomorphic mappings with a property of their ranges.

\begin{corollary}
Let $1\leq p<\infty$. Every $p$-integral holomorphic mapping $f\colon U\to F$ has relatively weakly compact range. 
\end{corollary}

\begin{proof}
Let $f\in\I_p^{\H^\infty}(U,F)$. Then $T_f\in\I_p(\G^\infty(U),F)$ by Proposition \ref{integral}, hence $T_f\in\W(\G^\infty(U),F)$ by \cite[Proposition 5.5 and Theorem 2.17]{DisJarTon-95}, and thus $f\in\H^\infty_\W(U,F)$ by \cite[Proposition 3.4]{Muj-91}.
\end{proof}

\subsection{$p$-nuclear holomorphic mappings}\label{subsection 4}

Given Banach spaces $E$, $F$ and $1\leq p<\infty$, we denote by $\Nu_p(E,F)$ the Banach space of all $p$-nuclear linear operators $T\colon E\to F$ with the norm 
$$
\nu_p(T)=\inf\left\{\left\|A\right\|\left\|M_\lambda\right\|\left\|B\right\|\right\},
$$
where the infimum is taken over all such $p$-nuclear factorizations $(A,M_\lambda,B)$ of $T$ in the form
$$
T=A\circ M_{\lambda}\circ B\colon E\stackrel{B}{\rightarrow}\ell_\infty\stackrel{M_{\lambda}}{\rightarrow}\ell_p\stackrel{A}{\rightarrow}F,
$$
where $A\in\L(\ell_p,F)$, $B\in\L(E,\ell_\infty)$ and $M_\lambda\in\L(\ell_\infty,\ell_p)$ is a diagonal operator induced by a sequence $\lambda\in\ell_p$ (see \cite[p. 111]{DisJarTon-95}).

In analogy with this concept, we introduce the following variant in the holomorphic setting.

\begin{definition}
Let $E,F$ be complex Banach spaces, $U$ be an open subset of $E$ and $1\leq p<\infty$. A mapping $f\colon U\to F$ is said to be $p$-nuclear holomorphic if there exist an operator $T\in\L(\ell_p,F)$, a mapping $g\in\H^\infty(U,\ell_\infty)$ and a diagonal operator $M_\lambda\in\L(\ell_\infty,\ell_p)$ induced by a sequence $\lambda\in\ell_p$ such that $f=T\circ M_\lambda\circ g$, that is, the following diagram commutes:
$$
\begin{tikzpicture}
  \node (U) {$U$};
  \node (F) [right of=U] {$F$};
  \node (li) [below of=U] {$\ell_\infty$};
  \node (lp) [below of=F] {$\ell_p$};
  \draw[->] (U) to node {$f$} (F);
  \draw[->] (U) to node [swap] {$g$} (li);
  \draw[->] (li) to node {$M_\lambda$} (lp);
  \draw[->] (lp) to node [swap] {$T$} (F);
\end{tikzpicture}
$$
The triple $(T,M_\lambda,g)$ is called a $p$-nuclear holomorphic factorization of $f$. We set 
$$
\nu^{\H^\infty}_p(f)=\inf\left\{\left\|T\right\|\left\|M_\lambda\right\|\left\|g\right\|_\infty\right\},
$$
where the infimum is extended over all such factorizations of $f$. Let $\Nu^{\H^\infty}_p(U,F)$ denote the set of all $p$-nuclear holomorphic mappings from $U$ into $F$.  
\end{definition} 


A study on $p$-nuclear holomorphic mappings similar to that of the preceding subsection on $p$-integral holomorphic mappings is developed next.

\begin{proposition}\label{nuclear}
Let $1\leq p<\infty$ and $f\in\H^\infty(U,F)$. Then $f\colon U\to F$ is $p$-nuclear holomorphic if and only if its linearization $T_f\colon\G^\infty(U)\to F$ is $p$-nuclear. In this case,
$$
\nu_p(T_f)=\nu_p^{\H^\infty}(f).
$$
Furthermore, the mapping $f\mapsto T_f$ is an isometric isomorphism from $(\Nu_p^{\H^\infty}(U,F),\nu_p^{\H^\infty})$ onto $(\Nu_p(\G^\infty(U),F),\nu_p)$. 
\end{proposition}

\begin{proof}
If $f\colon U\to F$ is $p$-nuclear holomorphic, then we have 
$$
f=T\circ M_{\lambda}\circ g\colon U\stackrel{g}{\rightarrow}\ell_\infty\stackrel{M_{\lambda}}{\rightarrow}\ell_p\stackrel{T}{\rightarrow}F,
$$
where $(T,M_\lambda,g)$ is a $p$-nuclear holomorphic factorization of $f$. Using Theorem \ref{teo0}, we obtain 
$$
T_f\circ g_U
=T\circ M_{\lambda}\circ T_g\circ  g_U\colon U\stackrel{ g_U}{\rightarrow}\G^\infty(U)\stackrel{T_g}{\rightarrow}\ell_\infty\stackrel{M_{\lambda}}{\rightarrow}\ell_p\stackrel{T}{\rightarrow}F.
$$
By the denseness of $\lin( g_U)$ in $\G^\infty(U)$, we deduce that
$$
T_f=T\circ M_{\lambda}\circ T_g\colon\G^\infty(U)\stackrel{T_g}{\rightarrow}\ell_\infty\stackrel{M_{\lambda}}{\rightarrow}\ell_p\stackrel{T}{\rightarrow}F,
$$
and therefore $T_f\colon\G^\infty(U)\to F$ is $p$-nuclear. Furthermore, we have
$$
\nu_p(T_f)\leq\left\|T\right\|\left\|M_{\lambda}\right\|\left\|T_g\right\|=\left\|T\right\|\left\|M_{\lambda}\right\|\left\|g\right\|_\infty
$$
and since we were working with an arbitrary $p$-nuclear holomorphic factorization for $f$, we obtain
$$
\nu_p(T_f)\leq\nu_p^{\H^\infty}(f).
$$
Conversely, if $T_f\colon\G^\infty(U)\to F$ is $p$-nuclear, we have 
$$
T_f=T\circ M_{\lambda}\circ S\colon\G^\infty(U)\stackrel{S}{\rightarrow}\ell_\infty\stackrel{M_{\lambda}}{\rightarrow}\ell_p\stackrel{T}{\rightarrow}F,
$$
where $(T,M_\lambda,S)$ is a $p$-nuclear factorization of $T_f$. Note that $g:=S\circ g_U\in\H^\infty(U,\ell_\infty)$ and since  
$$
f=T\circ M_\lambda\circ g\colon U\stackrel{g}{\rightarrow}\ell_\infty\stackrel{M_{\lambda}}{\rightarrow}\ell_p\stackrel{T}{\rightarrow}F,
$$ 
we conclude that $f$ is $p$-nuclear holomorphic. Furthermore, we infer that 
$$
\nu_p^{\H^\infty}(f)\leq\left\|T\right\|\left\|M_{\lambda}\right\|\left\|g\right\|_\infty=\left\|T\right\|\left\|M_{\lambda}\right\|\left\|S\right\|,
$$
and this ensures that 
$$
\nu_p^{\H^\infty}(f)\leq\nu_p(T_f).
$$
To prove the last assertion of the statement, it suffices to show that the mapping $f\mapsto T_f$ from $\Nu_p^{\H^\infty}(U,F)$ to $\Nu_p(\G^\infty(U),F)$ is surjective. Take $T\in\Nu_p(\G^\infty(U),F)$ and then $T=T_f$ for some $f\in\H^\infty(U,F)$ by Theorem \ref{teo0}. Hence $T_f\in\Nu_p(\G^\infty(U),F)$ and this implies that $f\in\Nu_p^{\H^\infty}(U,F)$ by the above proof. 
\end{proof}

Using Theorem \ref{ideal} and Proposition \ref{nuclear}, we obtain the following.

\begin{corollary}
Let $1\leq p<\infty$. Then $\Nu_p^{\H^\infty}=\Nu_p\circ\H^\infty$ and $\nu_p^{\H^\infty}(f)=\left\|f\right\|_{\Nu_p\circ\H^\infty}$ for all $f\in\Nu_p^{\H^\infty}$. In particular, $\left(\Nu^{\H^\infty}_p,\nu^{\H^\infty}_p\right)$ is a Banach ideal of bounded holomorphic mappings. $\hfill\qed$
\end{corollary}

The following result follows from Proposition \ref{nuclear} and \cite[Corollary 5.24 (b)]{DisJarTon-95}.

\begin{corollary}
Let $1\leq p<q<\infty$. Then $\Nu^{\H^\infty}_p(U,F)\subseteq\Nu^{\H^\infty}_q(U,F)$ and $\nu^{\H^\infty}_q(f)\leq\nu^{\H^\infty}_p(f)$ for each $f\in\Nu^{\H^\infty}_p(U,F)$. $\hfill\qed$
\end{corollary}

\begin{corollary}
Let $1\leq p<\infty$. Every $p$-nuclear holomorphic mapping $f\colon U\to F$ has relatively compact range.
\end{corollary}

\begin{proof}
Let $f\in\Nu_p^{\H^\infty}(U,F)$. Then $T_f\in\Nu_p(\G^\infty(U),F)$ by Proposition \ref{nuclear}, hence $T_f\in\K(\G^\infty(U),F)$ by \cite[Corollary 5.24 (a)]{DisJarTon-95}, and thus $f\in\H^\infty_\K(U,F)$ by \cite[Proposition 3.4]{Muj-91}.
\end{proof}

As in the linear case \cite[Theorem 5.27]{DisJarTon-95} and in the Lipschitz case \cite[Theorem 2.12]{Saa-17}, $p$-nuclear holomorphic mappings admit the following factorization.

\begin{corollary}\label{nuclear-integral}
Let $1\leq p<\infty$ and $f\in\H^\infty(U,F)$. Then $f\in\Nu_p^{\H^\infty}(U,F)$ if and only if there exist a Banach space $G$, an operator $T\in\K(G,F)$ and a mapping $g\in\I_p^{\H^\infty}(U,G)$ such that $f=T\circ g$. 
\end{corollary}

\begin{proof}
If $f\in\Nu_p^{\H^\infty}(U,F)$, then $T_f\in\Nu_p(\G^\infty(U),F)$ by Proposition \ref{nuclear}. Then Theorem 5.27 in \cite{DisJarTon-95} shows that there exist a complex Banach space $G$, an operator $T\in\K(G,F)$ and an operator $S\in\I_p(G^\infty(U),G)$ such that $T_f=T\circ S$. Hence we have $f=T_f\circ g_U=T\circ S\circ g_U=T\circ g$, where $g=S\circ g_U\in\I_p^{\H^\infty}(U,F)$ by Corollary \ref{now}. 

Conversely, if there exist a Banach space $G$, an operator $T\in\K(G,F)$ and a mapping $g\in\I_p^{\H^\infty}(U,G)$ such that $f=T\circ g$, then $T_f\circ g_U=T\circ T_g\circ g_U$ which gives $T_f=T\circ T_g$ where $T_g\in\I_p(G^\infty(U),F)$ by Proposition \ref{integral}. Hence $T_f\in\Nu_p(G^\infty(U),F)$ by \cite[Theorem 5.27]{DisJarTon-95}, and so $f\in\Nu_p^{\H^\infty}(U,F)$ by Proposition \ref{integral}.
\end{proof}

Next, we study the inclusion relationships between the new classes of bounded holomorphic mappings considered. In a clear parallel to the linear case, we have the following.

\begin{corollary}
Let $1\leq p<\infty$. 
\begin{enumerate}
	\item $\Nu^{\H^\infty}_p(U,F)\subseteq\I^{\H^\infty}_p(U,F)$ and $\iota^{\H^\infty}_p(f)\leq\nu^{\H^\infty}_p(f)$ for all $f\in\Nu^{\H^\infty}_p(U,F)$.
	\item If $F$ is finite-dimensional, then $\Nu^{\H^\infty}_p(U,F)=\I^{\H^\infty}_p(U,F)$ with $\nu^{\H^\infty}_p(f)=\iota^{\H^\infty}_p(f)$ for all $f\in\Nu^{\H^\infty}_p(U,F)$.
\end{enumerate}
\end{corollary}

\begin{proof}
(1) If $f\in\Nu^{\H^\infty}_p(U,F)$, then $T_f\in\Nu_p(\G^\infty(U),F)$ with $\nu_p(T_f)=\nu_p^{\H^\infty}(f)$ by Proposition \ref{nuclear}. Since $\Nu_p(\G^\infty(U),F)\subseteq\I_p(\G^\infty(U),F)$ with $\iota_p(T)\leq\nu_p(T)$ for all $T\in\Nu_p(\G^\infty(U),F)$ by \cite[Corollary 5.24 (c)]{DisJarTon-95}, it follows that $T_f\in\I_p(\G^\infty(U),F)$. Hence $f\in\I^{\H^\infty}_p(U,F)$ with $\iota_p(T_f)=\iota_p^{\H^\infty}(f)$ by Proposition \ref{integral}, and further $\iota^{\H^\infty}_p(f)=\iota_p(T_f)\leq\nu_p(T_f)=\nu^{\H^\infty}_p(f)$.

(2) Assume that $F$ is finite-dimensional. If $f\in\I^{\H^\infty}_p(U,F)$, then $T_f\in\I_p(\G^\infty(U),F)$ with $\iota_p(T_f)=\iota_p^{\H^\infty}(f)$ by Proposition \ref{integral}. Hence $T_f\in\Nu_p(\G^\infty(U),F)$ with $\nu_p(T_f)=\iota_p(T_f)$ by \cite[Theorem 5.26]{DisJarTon-95}. It follows that $f\in\Nu^{\H^\infty}_p(U,F)$ with $\nu^{\H^\infty}_p(f)=\nu_p(T_f)$ by Proposition \ref{nuclear}, and so  $\nu^{\H^\infty}_p(f)=\iota^{\H^\infty}_p(f)$.
\end{proof}

In general, a bounded-holomorphic ideal $\I^{\H^\infty}$ does not coincide with $\I\circ\H^\infty$ as we see below. 

\begin{examples}
The ideal $\H_w^\infty$ of locally weakly compact bounded holomorphic mappings does not coincide with $\W\circ\H^\infty$. For example, let $\interior{\D}$ be the open unit disc in $\mathbb{C}$, and let $f\colon\interior{\D}\to c_0$ be the mapping defined by $f(z)=(z^n)_{n=1}^\infty$. By Example 3.2 in \cite{Muj-91}, $f$ is in $\H_k^\infty(\interior{\D},c_0)$ but $f$ is not in $\H_\W^\infty(\interior{\D},c_0)$. Hence $T_f$ fails to belong to $\W(\G^\infty(\interior{\D}),c_0)$ by \cite[Proposition 3.4 (b)]{Muj-91}. So by Theorem \ref{ideal} $f$ is not in $\W\circ\H^\infty$. The same example shows that in general $\H_k^\infty\neq\K\circ\H^\infty$ (see Example 3.3 in \cite{AroBotPelRue-10}).
\end{examples}

\section{Dual ideal of bounded holomorphic mappings}\label{section 4}

According to \cite[Section 4.4]{Pie-80}, the dual ideal of an operator ideal $\I$ is defined by 
$$
\I^\mathrm{dual}(E,F)=\left\{T\in\L(E,F)\colon T^*\in\I(F^*,E^*)\right\},
$$
where $E$ and $F$ are Banach spaces and $T^*$ is the adjoint operator of $T$. It is well known that $\I^\mathrm{dual}$ is also an operator ideal. Moreover, if $(\I,\left\|\cdot\right\|_\I)$ is a normed or Banach operator ideal, then $\I^\mathrm{dual}$ is so equipped with the norm
$$
\left\|T\right\|_{\I^\mathrm{dual}}=\left\|T^*\right\|_\I.
$$

In order to introduce the concept of bounded-holomorphic dual of an operator ideal, we will first need a holomorphic variant of the concept of adjoint operator between Banach spaces. 

\begin{definition}\cite{AroSch-76,Rya-88}
Let $E,F$ be complex Banach spaces and $U$ be an open subset of $E$. The transpose of a bounded holomorphic mapping $f\colon U\to F$ is the mapping $f^t\colon F^*\to\H^\infty(U)$ defined by 
$$
f^t(y^*)=y^*\circ f\qquad (y^*\in F^*).
$$
\end{definition}

It is easy to show (see, for example, \cite[Proposition 1.11]{JimRuiSep-22}) that $f^t$ is a continuous linear operator with $||f^t||=\left\|f\right\|_\infty$. Moreover, $f^t=J_U^{-1}\circ (T_f)^*$, where $J_U\colon\H^\infty(U)\to\G^\infty(U)^*$ is the isometric isomorphism defined in Theorem \ref{teo0}.

\begin{definition}
Given an operator ideal $\I$, the bounded-holomorphic dual of $\I$ is the set 
$$
(\I^{\H^\infty})^\mathrm{dual}(U,F)=\left\{f\in\H^\infty(U,F)\colon f^t\in\I(F^*,\H^\infty(U))\right\}.
$$
If $(\I,\left\|\cdot\right\|_\I)$ is a normed operator ideal and $f\in(\I^{\H^\infty})^\mathrm{dual}$, define
$$
\left\|f\right\|_{(\I^{\H^\infty})^\mathrm{dual}}=\left\|f^t\right\|_\I.
$$
\end{definition}

Next result assures that the bounded holomorphic mappings belonging to the bounded-holomorphic dual of an operator ideal $\I$ are exactly those that factorize through $\I^\mathrm{dual}$. The proof of this fact mimics the proofs of its polynomial version \cite[Theorem 2.2]{BotCalMor-14} and its Lipschitz version \cite[Theorem 3.9]{AchRueSanYah-16}.

\begin{theorem}\label{teo-dual}
Let $\I$ be an operator ideal. Then the transpose $f^t$ of a bounded holomorphic mapping $f$ belongs to $\I$ if and only if $f$ admits a factorization $f=T\circ g$ where $g$ is a bounded holomorphic mapping and the adjoint operator $T^*$ of the bounded linear operator $T$ belongs to $\I$, that is, 
$$
(\I^{\H^\infty})^\mathrm{dual}=\I^\mathrm{dual}\circ\H^\infty.
$$
Moreover, if $(\I,\left\|\cdot\right\|_\I)$ is a normed operator ideal, then
$$
\left\|f\right\|_{(\I^{\H^\infty})^\mathrm{dual}}=\left\|f\right\|_{\I^\mathrm{dual}}\circ\H^\infty
$$
for all $f\in(\I^{\H^\infty})^\mathrm{dual}$.
\end{theorem}

\begin{proof}
Let $f\in(\I^{\H^\infty})^\mathrm{dual}(U,F)$. Then $f\in\H^\infty(U,F)$ and $f^t\in\I(F^{*},\H^\infty(U))$. By Theorem \ref{teo0}, there exists $T_f\in\L(\G^\infty(U),F)$ such that $f=T_f\circ g_U$. Since $(T_f)^*=J_U\circ f^t\in\I(F^*,\G^\infty(U)^*)$, the ideal property of $\I$ yields that $T_f\in\I^\mathrm{dual}(\G^{\infty}(U),F)$. Hence $f\in\I^\mathrm{dual}\circ\H^\infty(U,F)$ with $\left\|f\right\|_{\I^\mathrm{dual}}\circ\H^\infty=\left\|T_f\right\|_{\I^\mathrm{dual}}$ by Theorem \ref{ideal}, and this proves the inclusion 
$$
(\I^{\H^\infty})^\mathrm{dual}(U,F)\subseteq\I^\mathrm{dual}\circ\H^\infty(U,F).
$$
Furthermore, we have 
\begin{align*}
\left\|f\right\|_{\I^\mathrm{dual}\circ\H^\infty}&=\left\|T_f\right\|_{\I^\mathrm{dual}}=\left\|(T_f)^*\right\|_{\I}=\left\|J_U\circ f^t\right\|_\I\\
                                          &\leq\left\|J_U\right\|\left\|f^t\right\|_{\I}=\left\|f^t\right\|_{\I}=\left\|f\right\|_{(\I^{\H^\infty})^\mathrm{dual}}.
\end{align*}
Conversely, let $f\in\I^\mathrm{dual}\circ\H^\infty(U,F)$. Then there are a complex Banach space $G$, a mapping $g\in\H^\infty(U,G)$ and an operator $T\in\I^\mathrm{dual}(G,F)$ such that $f=T\circ g$. Given $y^*\in F^*$, we have
$$
f^t(y^*)=(T\circ g)^t(y^*)=y^*\circ(T\circ g)=(y^*\circ T)\circ g=T^*(y^*)\circ g=g^t(T^*(y^*))=(g^t\circ T^*)(y^*),
$$
and thus $f^t=g^t\circ T^*$. Since $T^*\in\I(F^*,G^*)$ and $g^t\in\L(G^*,\H^\infty(U))$, we obtain that $f^t\in\I(F^*,\H^\infty(U))$. Hence $f\in(\I^{\H^\infty})^\mathrm{dual}(U,F)$ and this shows that 
$$
\I^\mathrm{dual}\circ\H^\infty(U,F)\subseteq(\I^{\H^\infty})^\mathrm{dual}(U,F).
$$
Moreover, we have  
\begin{align*}
\left\|f\right\|_{(\I^{\H^\infty})^\mathrm{dual}}&=\left\|f^t\right\|_{\I}=\left\|g^t\circ T^*\right\|_{\I}\\
                                             &\leq\left\|g^t\right\|\left\|T^*\right\|_{\I}=\left\|g\right\|_\infty\left\|T\right\|_{\I^\mathrm{dual}},
\end{align*}
and taking the infimum over all representations $T\circ g$ of $f$, we conclude that 
$$
\left\|f\right\|_{(\I^{\H^\infty})^\mathrm{dual}}\leq\left\|f\right\|_{\I^\mathrm{dual}\circ\H^\infty}.
$$
\end{proof}

Theorems 2.1, 2.2, 2.6 and 2.7 in \cite{JimRuiSep-22} can be deduced from our preceding results.

\begin{corollary}
Let $\I=\F,\overline{\F},\K,\W$. Then a bounded holomorphic mapping belongs to $\H^\infty_{\I}(U,F)$ if and only if its transpose belongs to $\I(F^*,\H^\infty(U))$.  
\end{corollary}

\begin{proof}
We have 
$$
\H^\infty_{\I}=\I\circ\H^\infty=\I^\mathrm{dual}\circ\H^\infty=(\I^{\H^\infty})^\mathrm{dual},
$$
where the first equality follows from Proposition \ref{new}, the second from \cite[Proposition 4.4.7]{Pie-80} and the third from Theorem \ref{teo-dual}.
\end{proof}

With a proof similar to the above but replacing \cite[Proposition 4.4.7]{Pie-80} by \cite[Theorem 5.15]{DisJarTon-95}, we obtain the following result on 1-integral holomorphic mappings.

\begin{corollary}
A bounded holomorphic mapping belongs to $\I_1^{\H^\infty}(U,F)$ if and only if its transpose belongs to $\I_1(F^*,\H^\infty(U))$. $\hfill\qed$
\end{corollary}

Regarding dual ideals, general representation theorems using topological tensor products provide useful tools for the linear case, even for the Lipschitz and multilinear cases. For future
research, it would be interesting to study a holomorphic analog for this kind of dual representations. It could provide another point of view more in the direction of the Defant--Floret book \cite{DefFlo-93} focusing on tensor products, which could open the door to more general approaches.\\





\textbf{Acknowledgements.} The authors would like to thank the referee for her (his) valuable comments and advice that have improved considerably the first version of this paper. This research was partially supported by project UAL-FEDER grant UAL2020-FQM-B1858, by Junta de Andaluc\'{\i}a grants P20$\_$00255 and FQM194, and by grant PID2021-122126NB-C31 funded by MCIN/AEI/ 10.13039/501100011033 and by ``ERDF A way of making Europe''.



\begin{thebibliography}{99}
\bibitem{AchRueSanYah-16} D. Achour, P. Rueda, E. A. S\'anchez-P\'erez and R. Yahi, Lipschitz operator ideals and the approximation property, J. Math. Anal. Appl. \textbf{436} (2016), no. 1, 217--236.
\bibitem{AroBotPelRue-10} R. Aron, G. Botelho, D. Pellegrino and P. Rueda, Holomorphic mappings associated to composition ideals of polynomials, Atti Accad. Naz. Lincei Rend. Lincei Mat. Appl. \textbf{21} (2010), no. 3, 261--274. 
\bibitem{AroRue-12} R. Aron and P. Rueda, Ideals of homogeneous polynomials, Publ. Res. Inst. Math. Sci. \textbf{48} (2012), no. 4, 957--969.
\bibitem{AroSch-76} R. M. Aron and M. Schottenloher, Compact holomorphic mappings on Banach spaces and the approximation property, J. Funct. Anal. \textbf{21} (1976), 7--30. 
\bibitem{BawGup-22} D. Baweja and M. Gupta, Characterizations of approximation properties defined by operator ideals in weighted Banach spaces of holomorphic functions, Adv. Oper. Theory \textbf{7} (2022), no. 3, Paper No. 39, 11 pp.
\bibitem{BotCalMor-14} G. Botelho, E. \c{C}ali\c{s}kan and G. Moraes, The polynomial dual of an operator ideal, Monatsh. Math. \textbf{173} (2014), no. 2, 161--174. 
\bibitem{BotFavMuj-19} G. Botelho, V. F\'avaro and J. Mujica, Absolute Schauder decompositions and linearization of holomorphic mappings of bounded type, Publ. Res. Inst. Math. Sci. \textbf{55} (2019), no. 3, 651--664.
\bibitem{BotPelRue-07} G. Botelho, D. Pellegrino and P. Rueda, On composition ideals of multilinear mappings and homogeneous polynomials, Publ. Res. Inst. Math. Sci. \textbf{43} (2007), no. 4, 1139--1155.
\bibitem{BotWoo-23} G. Botelho and R. Wood, Hyper-ideals of Multilinear Operators and Two-Sided Polynomial Ideals Generated by Sequence Classes, Mediterr. J. Math. \textbf{20} (2023), no. 1, Paper No. 35.
\bibitem{DefFlo-93} A. Defant and K. Floret, Tensor Norms and Operator Ideals, North-Holl. Math. Stud., vol. 176, North-Holland Publishing Co., Amsterdam, 1993.
\bibitem{DisJarTon-95} J. Diestel, H. Jarchow and A. Tonge, Absolutely Summing Operators, Cambridge Univ. Press, Cambridge, 1995.
\bibitem{FloGar-03} K. Floret and D. Garc\'ia, On ideals of polynomials and multilinear mappings between Banach spaces, Arch. Math. (Basel), \textbf{81} (3) (2003), pp. 300--308
\bibitem{GonGut-00} M. Gonz\'alez and J. M. Guti\'errez, Surjective factorization of holomorphic mappings, Comment. Math. Univ. Carolin. \textbf{41} (2000), no. 3, 469--476.
\bibitem{JimRuiSep-22} A. Jim{\'e}nez-Vargas, D. Ruiz-Casternado and J. M. Sepulcre, On holomorphic mappings with compact type range, Bull. Malays. Math. Sci. Soc. \textbf{46} (2023), no. 1, Paper No. 20, 16 pp. 
\bibitem{Lin-89} M. Lindstr\"om, On compact and bounding holomorphic mappings, Proc. Amer. Math. Soc. \textbf{105} (1989), no. 2, 356--361.
\bibitem{Muj-86} J. Mujica, Complex Analysis in Banach spaces, Dover Publications, 2010.
\bibitem{Muj-91} J. Mujica, Linearization of bounded holomorphic mappings on Banach spaces, Trans. Amer. Math. Soc. \textbf{324} (1991), 867--887. 
\bibitem{Pie-83} A. Pietsch, Ideals of multilinear functionals (designs of a theory), in: Proc. Second Int. Conf. on Operator Algebras, Ideals and Their Applications in Theoretical Physics,
Teubner-Texte Math. 67, Leipzig, 1983, 185–199.
\bibitem{Pie-80} A. Pietsch, Operator ideals, North-Holland Mathematical Library, vol. 20, North-Holland Publishing Co., Amsterdam-New York, 1980. Translated from German by the author.
\bibitem{Rob-92} N. Robertson, Asplund operators and holomorphic maps, Manuscripta Math. \textbf{75} (1992), no. 1, 25--34.
\bibitem{Rya-88} R. Ryan, Weakly compact holomorphic mappings on Banach spaces, Pacific J. Math. \textbf{131} (1988), no. 1, 179--190.
\bibitem{Saa-17} K. Saadi, On the composition ideals of Lipschitz mappings, Banach J. Math. Anal. \textbf{11} (2017), no. 4, 825--840.
\end{thebibliography}
\end{document}